\documentclass{amsart}
\usepackage{amsmath,amssymb,amsthm,amsfonts,enumerate,array, color,lscape,fancyhdr,layout,pst-all}
\usepackage[all,cmtip]{xy}
\usepackage[english]{babel}
\usepackage[latin1]{inputenc}
\usepackage{young}
\usepackage{tikz}
\usepackage{graphicx}
\usepackage{float}
\numberwithin{figure}{section}
\newcounter{numberofremark}
\newcommand\nothing[1]{}
\usepackage[active]{srcltx}
\usepackage[hyperindex,bookmarksnumbered,plainpages]{hyperref}

\newcommand{\dcl}{\DeclareMathOperator}
\dcl\bb{\mathfrak{b}}    
\dcl\CC{\mathcal{C}}    
\dcl\gl{\mathfrak{gl}}    
\dcl\g{\mathfrak{g}} 
\dcl\h{\mathfrak{h}} 
\dcl\n{\mathfrak{n}} 
\dcl\RR{\mathcal{R}}    
\dcl\Span{span} 
\dcl\U{\mathrm{U}}
\dcl\V{\mathfrak{V}}
\newlength\yStones
\newlength\xStones
\newlength\xxStones

\makeatletter
\def\Stones{\pst@object{Stones}}
\def\Stones@i#1{%
  \pst@killglue%
  \begingroup%
  \use@par%
  \setlength\xxStones{\xStones}%
  \expandafter\Stones@ii#1,,\@nil
  \endgroup
  \global\addtolength\xStones{0.6cm}%
  \global\addtolength\yStones{-7.5mm}}%
\def\Stones@ii#1,#2,#3\@nil{%
1  \rput(\xxStones,\yStones){%
    \psframebox[framesep=0]{%
      \parbox[c][6mm][c]{11mm}{\makebox[11mm]{$#1$}}}}%
  \addtolength\xxStones{1.2cm}%
  \ifx\relax#2\relax\else\Stones@ii#2,#3\@nil\fi}
\makeatother

\def\Stone#1{\fbox{\makebox[7mm]{\strut#1}}\kern2pt}

\newtheorem{theorem}{Theorem}[section]
\newtheorem{lemma}[theorem]{Lemma}
\newtheorem{corollary}[theorem]{Corollary}
\newtheorem{proposition}[theorem]{Proposition}
\newtheorem{example}[theorem]{Example}
\newtheorem{remark}[theorem]{Remark}
\newtheorem{notation}[theorem]{Notation}

\newtheorem{definition}[theorem]{Definition}

\def\R{\mathbb{R}}

\def\Z{\mathbb{Z}}

\makeindex

\begin{document}
\title[Faces of polyhedra associated with relation modules]{Faces of polyhedra associated with relation modules}

\author[G. Benitez]{Germ\'an Benitez}
\address{Universidade Federal do Amazonas \\ Manaus AM \\ Brazil \&  Instituto de Matem\'atica e Estat\'istica \\ Universidade de S\~ao Paulo \\ Sao Paulo SP \\ Brazil}
\email{gabm@ufam.edu.br}

\author[L. E. Ramirez]{Luis Enrique Ramirez}
\address{Universidade Federal do ABC \\ Santo Andr\'e SP \\ Brazil} 
\email{luis.enrique@ufabc.edu.br}

%===========================================================%
						%ABSTRACT%
%===========================================================%

\begin{abstract}

Relation Gelfand--Tsetlin $\mathfrak{gl}_n$-modules were introduced in \cite{FRZ19}, and are determined by some special directed graphs and Gelfand-Tsetlin characters. In this work we constructed polyhedra associated with the class of relation modules, which includes as a particular case, any classical Gelfand--Tsetlin polytope.  Following the ideas presented in \cite{LM04} we give a characterization of $d$-faces of the associated polyhedra in terms of a matrix related to the corresponding graph. 

\end{abstract}


\maketitle

%===========================================================%
						%TABLE OF CONTENTS%
%===========================================================%

%\tableofcontents

%===========================================================%
						%INTRODUCTION%
%===========================================================%

\section{Introduction}

 	Gelfand-Tsetlin theorem \cite{GT50}, is one of the most remarkable results in representation theory and  gives an explicit realization of any simple finite dimensional $\mathfrak{gl}_n$-module. Their construction includes an explicit basis of tableaux with entries satisfying certain betweenness conditions. Using the betweenness conditions and highest weights, Berenstein and Zelevinsky \cite{BZ89} constructed the so--called \emph{Gelfand--Tsetlin polytopes}, and provide an effective way to relate representation theory of the general Lie algebra $\mathfrak{gl}_n$ and polyhedra. The main property of Gelfand--Tsetlin polytopes is a relation between the number of integral points and the weight multiplicities of the corresponding finite dimensional modules.  In 2004, De Loera and  McAllister in \cite{LM04} get to characterize the points of $d$-face (points in a face of dimension $d$) of the Gelfand--Tsetlin polytopes.
 
	In \cite{FRZ19}, Futorny, Ramirez and Zhang introduced \emph{sets of relations} as an attempt to generalize the betweenness conditions that characterize basis elements in Gelfand--Tsetlin Theorem. Associated with sets of relations the authors constructed  explicit modules called \emph{relation Gelfand--Tsetlin $\mathfrak{gl}_n$-modules,}  \emph{relation $\mathfrak{gl}_n$-modules}, or \emph{relation modules}. Using as a motivation the results in \cite{LM04}, and \cite{FRZ19}, the concept of Gelfand--Tsetlin polyhedra associated with relation modules was introduced in \cite{Car19}, and some well behave cases were studied. 

	Construction of polyhedra associated to a directed graph is well known, however, in this paper we focus on special graphs such that the associated polyhedra preserves the relationship between the number of ``integral points'' and the dimensions of the weight spaces of the corresponding relation module. Our approach includes a combinatorial characterization for the dimension of the minimal face associated to a point of the polyhedron.
  
  	In the literature,  Gelfand--Tsetlin patterns appear in connection with polyhedra in \cite{BZ89,KB95,LM04, Pos09,ABS11}, and probability in \cite{WW09}. In \cite{BZ89,KB95,LM04}, Gelfand--Tsetlin patterns are defined to be triangular arrangements of non-negative integers subject to some order restrictions, namely $\mathcal{C}_1$ (see Example \ref{relac}). There are many different ways to define Gelfand--Tsetlin polytopes and polyhedra. In \cite[p. 92]{KB95}, and \cite[Definition 1.2., p. 460]{LM04} are certain subsets of $\mathbb{R}_{\geq 0}^{\frac{n(n+1)}{2}}$. In \cite{ABS11,GKT13} the authors fixed the $n$th row and do not require the entries to be non--negative. In contrast, Danilov, Karzanov and  Koshevoy, in \cite{DKK05}, study certain polyhedra related to an extension of Gelfand--Tsetlin patterns, in the sense that it is not necessary a triangular configuration. In this paper we fix $\mathfrak{gl}_n$ and we will define $\mathcal{C}$-pattern to be a triangular arrangement of complex numbers subject to some order restrictions $\mathcal{C}$. Our definition of $\mathcal{C}$-pattern is also an extension of Gelfand--Tsetlin patterns associated with $\mathfrak{gl}_n$.	

	This paper is organized as follows. In Section~\ref{sec:gln-Mod} we discuss some preliminaries on highest weight modules and Gelfand--Tsetlin modules. Section~\ref{sec:rel_mod} is dedicated to relation modules, to this goal we recall the notions of set of relations $\mathcal{C}$, its associated graph, and Gelfand--Tsetlin tableaux. Associated with a set of relations $\mathcal{C}$, in Section~\ref{sec:GRelat}, we introduce $\mathcal{C}$-pattern and three kind of polyhedra associated with $\mathcal{C}$, we exhibit the relation of these polyhedra with relation modules. Also, we define a tiling and tiling matrix associated with a $\mathcal{C}$-pattern, and study some of their properties. In Section~\ref{sec:main} we use the results of the previous section to obtain a combinatorial characterization for the dimension of a minimal face with respect to a point in the polyhedron.

%===========================================================%
					%RELATION MODULES%
%===========================================================%

\section{Preliminaries on modules}
\label{sec:gln-Mod}

	Unless otherwise specified, the ground field will be $\mathbb{C}$, $\mathfrak{g}$ will denote a finite dimensional Lie algebra with a Cartan subalgebra $\mathfrak{h}$, and a triangular decomposition $\mathfrak{g}=\mathfrak{n}^-\oplus\mathfrak{h}\oplus\mathfrak{n}^+$. By $U(\mathfrak{g})$ we denote the universal enveloping algebra of  $\mathfrak{g}$ and by $V^{*}$ the dual space $\textup{Hom}_{\mathbb{C}}(V,\mathbb{C})$ of a vector space $V$.

%___________________________________________________________%

\subsection{Highest weight modules}

We begin this section recalling some definitions and basic results on weight modules. For more details we refer the reader to \cite{Hum08}.

%-----------------------------------------------------------%

\begin{definition}
 Let $M$ be an $U(\mathfrak{g})$-module. For each $\lambda\in\mathfrak{h}^{*}$, the $\lambda$-\emph{weight space} is defined by $
	M_{\lambda}:=\left\{v\in M\mid h\cdot v=\lambda(h)v,\mbox{ for all  }h\in\mathfrak{h}\right\}.
	$ If $M_{\lambda}\neq 0$, we say that $\lambda$ is a \emph{weight of} $M$, and $\dim (M_{\lambda})$ will be called  \emph{multiplicity} of $\lambda$. $M$ is called \emph{weight module} if it is equal to the direct sum of its weight spaces.
\end{definition}

%-----------------------------------------------------------%

A $\mathfrak{g}$-module $M$ is a \emph{highest weight module} if there exist  $\lambda\in\mathfrak{h}^{*}$, and a nonzero vector $v^+\in M_{\lambda}$ such that $M=U(\mathfrak{g})\cdot v^+$ and $U(\mathfrak{n}^{+})\cdot v^{+}=0$. The vector $v^+$ is called a \emph{highest weight vector} of $M$ and $\lambda$  a \emph{highest weight} of $M$. 

It is a well known result in representation theory that simple finite-dimensional $\mathfrak{gl}_n$-modules are highest weight modules, and there is an one to one correspondence between simple finite dimensional modules and integral dominant $\mathfrak{gl}_n$-weights (i.e. $n$-tuples $\bar{\lambda}=(\lambda_{1},\ldots,\lambda_{n})\in\mathbb{Z}^n/\langle (1,\ldots,1)\rangle$ satisfying $\lambda_{i}-\lambda_{i+1}\in \Z_{\geq 0}, \ \ \mbox{for all } i=1,2,\dots, n-1$). 

%___________________________________________________________%

\subsection{Gelfand--Tsetlin modules}

	From now on we fix  $\mathfrak{g}:=\mathfrak{gl}_n$ to be linear Lie algebra of $n\times n$ matrices over $\mathbb{C}$, the Cartan subalgebra $\mathfrak{h}$ of diagonal matrices, and the triangular decomposition  $\mathfrak{g}=\mathfrak{n}^-\oplus\mathfrak{h}\oplus\mathfrak{n}^+$, where $\mathfrak{n}^+$ is the subalgebra of upper triangular matrices and $\mathfrak{n}^-$  the subalgebra of lower triangular matrices.

	A large and important class of weight modules that have been extensively studied in the last 30 years is the so called \emph {Gelfand--Tsetlin modules.} This modules are characterized by a well behaved action of a fixed maximal commutative subalgebra of $U(\mathfrak{gl}_n)$, known as \emph{Gelfand-Tsetlin subalgebra}.  In what follows we give a formal definition of Gelfand-Tsetlin subalgebra and Gelfand-Tsetlin modules.

	For $m\leq n$, let $\mathfrak{gl}_{m}$ be the Lie subalgebra of $\mathfrak{gl}_{n}$ spanned by $\left\{ E_{ij}\mid i,j=1,\ldots,m \right\}$, where $E_{ij}$ is the $(i,j)$th  elementary matrix.  The \emph{(standard) \index{defs}{Gelfand--Tsetlin subalgebra}} ${\Gamma}$ of $U(\mathfrak{gl}_{n})$ (see \cite{DFO94}) is generated by $\left\{Z_m\mid m=1,\ldots, n \right\}$, where $Z_{m}$ stands for the center of $U(\mathfrak{gl}_{m})$.  
 
%-----------------------------------------------------------%

\begin{definition}
\label{Definition: GT-modules} 
	A finitely generated $U(\mathfrak{gl}_{n})$-module $M$ is called a \emph{Gelfand--Tsetlin module (with respect to
$\Gamma$)} if  $M|_{\Gamma}=\bigoplus_{\chi\in\Gamma^{*}}M_{\chi}$,  where 
	$$
	M_{\chi}=\left\{v\in M\mid \text{for each }\gamma\in\Gamma\text{, }\exists k\in\mathbb{Z}_{\geq 0} \text{ such that } (\gamma-\chi(\gamma))^{k}v=0 \right\}.
	$$
\end{definition}

%-----------------------------------------------------------%

\begin{remark} 
	Any weight module with finite dimensional weight spaces (in particular any module in category $\mathcal{O}$) is a Gelfand--Tsetlin module. Moreover, as $\mathfrak{h}\subseteq \Gamma$, any simple Gelfand-Tsetlin module is a weight module.

\end{remark}

%===========================================================%
					% RELATION MODULES %
%===========================================================%

\section{Relation modules}
\label{sec:rel_mod}
	The class of \emph{relation} Gelfand--Tsetlin modules was introduced in \cite{FRZ19}. These modules generalize the construction of simple finite-dimensional modules \cite{GT50} and generic Gelfand--Tsetlin modules \cite{DFO94}. We recall the construction and main properties.

	Set $\mathfrak{V}:=\left\{(i,j)\in\mathbb{Z}\times\mathbb{Z} \mid 1\leq j\leq i\leq n\right\}$, and  $\mathcal{R}:=\mathcal{R}^{-}\cup\mathcal{R}^{0}\cup\mathcal{R}^{+}\subseteq \mathfrak{V}\times\mathfrak{V}$, where  
	\begin{align*}
 	\mathcal{R}^+ &:=\{((i,j);(i-1,t))\ |\ 2\leq j\leq i\leq n,\ 1\leq t\leq i-1\},\\
 	\mathcal{R}^- &:=\{((i,j);(i+1,s))\ |\ 1\leq j\leq i\leq n-1,\ 1\leq s\leq i+1\},\\
 	\mathcal{R}^{0}&:=\{((n,i);(n,j))\ |\ 1\leq i\neq j\leq n\}.
 	\end{align*}
Any subset $\mathcal{C}\subseteq \mathcal{R}$ will be called a \emph{set of relations}. 

%-----------------------------------------------------------%

\begin{definition}
\label{def:graph}
	Let $\mathcal{C}$ be a set of relations.
	\begin{itemize}
	\item[(i)] By $G(\mathcal{C})$ we denote the directed graph with set of vertices $\mathfrak{V}$, and an arrow from vertex $(i,j)$ to $(r,s)$ if and only if $((i,j);(r,s))\in\mathcal{C}$. For convenience, we will picture the set of vertices $\mathfrak{V}$ as a triangular arrangement with $n$ rows, where the $k$th row is $\left\{(k,1),\dots,(k,k)\right\}$.

	\item[(ii)] By $\mathfrak{V}(\mathcal{C})$ we will denote the subset of $\mathfrak{V}$ of all vertices that are source or target of an arrow in $G(\mathcal{C})$.

 	\item[(iii)] Given $(i,j),(r,s)\in\mathfrak{V}$ we write $(i,j)\succeq_{\mathcal{C}} (r,s)$ if theres exist a directed path in $G(\mathcal{C})$ from $(i,j)$ to $(r,s)$.

	\item[(iv)] We call $\mathcal{C}$ \emph{reduced}, if for every $(k,j)\in\mathfrak{V}(\mathcal{C})$ the following conditions are satisfied:
		\begin{itemize}
		\item There exists at most one $i$ such that $((k,j);(k+1,i))\in\mathcal{C}$;
	
		\item There exists at most one $i$ such that $((k+1,i);(k,j))\in\mathcal{C}$;
	
		\item There exists at most one $i$ such that $((k,j);(k-1,i))\in\mathcal{C}$;
	
		\item There exists at most one $i$ such that $((k-1,i);(k,j))\in\mathcal{C}$;
	
		\item No relations in the top row follow from other relations.
	\end{itemize}

	\end{itemize}
\end{definition}

%-----------------------------------------------------------%

\begin{example}
\label{relac}
	For any $1\leq k\leq n$, we consider the following sets of relations: 
	\begin{center}
	\begin{tabular}{| c || l | }
	\hline 
	Set of relations  $\mathcal{C}$ & \ \ \ \ \ \ \ \ \ \ \ \ $\mathfrak{V}(\mathcal{C})$\\ 
	\hline\hline
 	$\mathcal{C}_{k}^{+}:=\left\{ ((i+1,j);(i,j)) \mid k\leq j\leq i\leq n-1\right\}$ & $\{(i,j)\in\mathfrak{V}\mid j\geq k\}\setminus\{(n,n)\}$\\ 
 	\hline
	$\mathcal{C}_{k}^{-}:=\left\{ ((i,j);(i+1,j+1)) \mid  k\leq j\leq i\leq n-1\right\}$& $\{(i,j)\in\mathfrak{V}\mid  j\geq k\}\setminus\{(n,k)\}$ \\ 
	\hline
	$\mathcal{C}_{k}:=\mathcal{C}_{k}^{+}\cup\mathcal{C}_{k}^{-}$ &$\begin{cases}\{(i,j)\in\mathfrak{V}\mid j\geq k\},& k<n\\
\emptyset,& k=n.
\end{cases}$ \\ 
	\hline
	\end{tabular}
	\end{center}

\noindent We will refer to $\mathcal{C}_1$ as the \emph{standard set of relations}, and to $\mathcal{C}_{n}=\emptyset$ as \emph{generic set of relations}. For $n=4$, the graphs associated with $\mathcal{C}_{1}^+$, $\mathcal{C}_1$, $\mathcal{C}_{2}^-$, and $\mathcal{C}_2$ are:
 
\noindent
\begin{minipage}[l]{0.5\textwidth}

	\begin{center}
	\tiny{\fbox{
	\xymatrixrowsep{0.3cm}\xymatrixcolsep{0.1cm}
	\xymatrix @C=0.01em {
	(4,1)\ar@[red][rd]& & (4,2)\ar@[red][rd]& &(4,3)\ar@[red][rd]& &(4,4)&\\
			          &(3,1)\ar@[red][rd]& &(3,2)\ar@[red][rd]& &(3,3)& &\\
					  & & (2,1)\ar@[red][rd]& &(2,2)& & &\\
					  & & &(1,1)& & & &}
	} }
	\end{center}

	\begin{center}	
	\tiny{\fbox{
	\xymatrixrowsep{0.3cm}\xymatrixcolsep{0.1cm}
	\xymatrix @C=0.01em {
	(4,1)& & (4,2) & &(4,3) & &(4,4)&\\
		 &(3,1)& &(3,2)\ar@[red][ur]& &(3,3)\ar@[red][ur]& &\\
		 & &(2,1)& &(2,2)\ar@[red][ur]& & &\\
		 & & &(1,1)& & & &}
	}}
	\end{center}

\end{minipage}
\begin{minipage}[r]{0.5\textwidth}
	\begin{center}
	\tiny{\fbox{
	\xymatrixrowsep{0.3cm}\xymatrixcolsep{0.1cm}
	\xymatrix @C=0.01em {
	(4,1)\ar@[red][rd]& & (4,2)\ar@[red][rd]& &(4,3)\ar@[red][rd]& &(4,4)&\\
			          &(3,1)\ar@[red][rd]\ar@[red][ur]& &(3,2)\ar@[red][rd]\ar@[red][ur]& &(3,3)\ar@[red][ur]& &\\
					  & &(2,1)\ar@[red][rd]\ar@[red][ur]& &(2,2)\ar@[red][ur]& & &\\
					  & & &(1,1)\ar@[red][ur]& & & &}
	}      }
	\end{center} 

	\begin{center}	
	\tiny{\fbox{
	\xymatrixrowsep{0.3cm}\xymatrixcolsep{0.1cm}
	\xymatrix @C=0.01em {
	(4,1)& & (4,2)\ar@[red][rd]& &(4,3)\ar@[red][rd]& &(4,4)&\\
		 &(3,1)& &(3,2)\ar@[red][rd]\ar@[red][ur]& &(3,3)\ar@[red][ur]& &\\
		 & &(2,1)& &(2,2)\ar@[red][ur]& & &\\
		 & & &(1,1)& & & &}
	}}
	\end{center}

\end{minipage}

\end{example}

%-----------------------------------------------------------%

\begin{remark}

Whenever we refer to connected components of a directed graph $G$, we are referring to the connected components of the graph obtained from $G$ forgetting the orientation of the arrows. 
\end{remark}

Entries of vectors $M\in \mathbb{C}^{\frac{n(n+1)}{2}}$ will be indexed by elements of $\mathfrak{V}$, ordered as  
	$
(m_{n1},\dots,m_{nn}| m_{n-1,1},\dots,m_{n-1,n-1}|\cdots| m_{21},m_{22}| m_{11}),
	$
and $T(M)$ will denote the triangular configuration of height $n$, with $k$th row $\left(m_{k1},\ldots,m_{kk}\right)$ for $k=1,\ldots, n$. We will refer to $T(M)$ as \emph{Gelfand--Tsetlin tableau}. For any $\mathbb{A}\subseteq \mathbb{C}$, we denote by $T_{n}(\mathbb{A})$ the set of all Gelfand--Tsetlin tableaux of height $n$ with entries in $\mathbb{A}$. By ${\mathbb Z}_0^\frac{n(n+1)}{2}$ we will denote the set of vectors $M$ in ${\mathbb Z}^\frac{n(n+1)}{2}$ such that $m_{ni}=0$ for $i=1,\ldots,n$. Given a Gelfand--Tsetlin tableau $T(L)$, elements of the set $T_{n}(L+{\mathbb Z}_{0}^\frac{n(n+1)}{2})$ will be called \emph{$L$-integral tableaux}.
	
%-----------------------------------------------------------%

\begin{definition}[{\cite[Definition 4.2]{FRZ19}}]
\label{def:C-realization}
Let $\mathcal{C}$ be a set of relations and $T(L)$ any Gelfand--Tsetlin tableau.
	\begin{itemize}
	\item[(i)] We say that \emph{$T(L)$ satisfies $\mathcal{C}$}, if $l_{ij}-l_{rs}\in \mathbb{Z}_{\geq 0}$  for any $((i,j); (r,s))\in \mathcal{C}$. 
	
	\item[(ii)] $T(L)$ is a \emph{$\mathcal{C}$-realization}, if $T(L)$ satisfies $\mathcal{C}$ and for any $1\leq k\leq n-1$ we have, $l_{ki}-l_{kj}\in \mathbb{Z} $ if and only if $(k,i)$ and $(k,j)$ belong to  the same connected component of $G(\mathcal{C})$.
 
 	\item[(iii)]  We call $\mathcal{C}$ \emph{noncritical} if for any $\mathcal{C}$-realization $T(M)$ one has $m_{ki}-m_{kj}+j-i\neq 0$, $1\leq k\leq n-1,\ i\neq j$, whenever  $(k,i)$ and $(k,j)$ are in the same connected component of $G(\mathcal{C})$.
 	
	\item[(iv)] Suppose that $T(L)$ satisfies $\mathcal{C}$. By ${\mathcal B}_{\mathcal{C}}(T(L))$ we denote the set of  $L$-integral tableaux satisfying $\mathcal{C}$, and  by $V_{\mathcal{C}}(T(L))$  the free vector space over $\mathbb{C}$ with basis ${\mathcal B}_{\mathcal{C}}(T(L))$. 
	\end{itemize}
\end{definition}
%-----------------------------------------------------------%

\begin{remark}
\label{rem:XvsT(X)} 
Whenever we use a Gelfand-Tsetlin tableau $T(X)$, we are considering it as an element of a free vector space of the form $V_{\mathcal{C}}(T(L))$ for some set of relations $\mathcal{C}$. In particular, we should be careful when comparing $X$ with $T(X)$. For instance, $T(X+Y)\neq T(X)+T(Y)$.  
\end{remark}

%-----------------------------------------------------------%

\begin{example}
\label{example realizations} 
Let $\mathcal{C}$ be the set of relations with associated graph $G(\mathcal{C})$, and Gelfand--Tsetlin tableaux $T(L)$ and $T(M)$ defined as follows: 
	\begin{center}
	\tiny{
	$G(\mathcal{C})$=	
	\begin{tabular}{c c c c c c c }
	\xymatrixrowsep{0.3cm}\xymatrixcolsep{0.1cm}
	\xymatrix @C=0.1em {
	&(4,1)& & (4,2)& &(4,3)& &(4,4)&\\
		& &(3,1)& &(3,2)\ar@[red][rd]& &(3,3)& &\\
   		  & & &(2,1)\ar@[red][rd]\ar@[red][ur]& &(2,2)& & &\\
				& & & &(1,1)\ar@[red][ur]& & & &}
	\end{tabular}} 
	\end{center}	

	\begin{center}
 	\hspace{1cm}\Stone{$1$}\Stone{$2$}\Stone{$3$}\Stone{$4$}\hspace{1cm}\Stone{$1$}\Stone{$2$}\Stone{$3$}\Stone{$4$}\\[0.2pt]
 	\hspace{1cm}\Stone{$1$}\Stone{$0$}\Stone{$2$}\hspace{2cm}\Stone{$\sqrt{2}$}\Stone{$1$}\Stone{$\sqrt{3}$}\\[0.2pt]
  	\hspace{-0.2cm}$T(L)$= \Stone{$1$}\Stone{$-2$} \hspace{1.7cm} $T(M)$=  \Stone{$2$}\Stone{$0$}\\[0.2pt]
 	\hspace{1cm}\Stone{$0$}\hspace{4.2cm}\Stone{$1$}\\\vspace{0.3cm}
	\end{center}
In this case $T(L)$ and $T(M)$ satisfy $\mathcal{C}$,  $T(M)$ is a $\mathcal{C}$-realization, and $T(L)$ is not a $\mathcal{C}$-realization (for instance, $l_{31}-l_{32}\in\mathbb{Z}$). Moreover,  as $(2,1)\succeq_{\mathcal{C}}  (2,2)$, for any $\mathcal{C}$-realization $T(S)$  we should have $s_{21}-s_{22}\in\mathbb{Z}_{\geq 0}$, in particular $\mathcal{C}$ is a noncritical set of relations. 
\end{example}

%-----------------------------------------------------------%

\begin{remark}
Definitions \ref{def:C-realization}\textup{(i),(iii)} are slightly different from the original definitions in \cite{FRZ19}, however they can be related using the shift $l_{kj}\longleftrightarrow l_{kj}-j+1$.
\end{remark}

%-----------------------------------------------------------%

\begin{definition}[{\cite[Definition 4.4]{FRZ19}}]
\label{def:adm}
Let $\mathcal{C}$ be a set of relations. We call  $\mathcal{C}$  \emph{admissible} if for any $\mathcal{C}$-realization $T(L)$,  $V_{\mathcal{C}}(T(L))$ has structure of a $\mathfrak{gl}_{n}$-module, endowed with the following action of the generators of $\mathfrak{gl}_{n}$ on any $T(M)\in \mathcal{B}_{\mathcal{C}}(T(L))$,
	\begin{align}\label{E_{k,k+1}}
	E_{k,k+1}(T(M))&=-\sum_{i=1}^{k}\left(\frac{\displaystyle\prod_{j=1}^{k+1}( m_{ki}- m_{k+1,j}+j-i)}{\displaystyle\prod_{j\neq i}^{k}( m_{ki}- m_{kj}+j-i)}\right)T(M+\delta^{ki}),\\\label{E_{k+1,k}}
	E_{k+1,k}(T(M))&=\sum_{i=1}^{k}\left(\frac{\displaystyle\prod_{j=1}^{k-1}( m_{ki}- m_{k-1,j}+j-i)}{\displaystyle\prod_{j\neq i}^{k}( m_{ki}- m_{kj}+j-i)}\right)T(M-\delta^{ki}),\\\label{E_{k,k}}
	E_{kk}(T(M))	   &=\left(\sum_{i=1}^{k} m_{ki}-\sum_{i=1}^{k-1} m_{k-1,i}\right)T(M),
	\end{align}
where $\delta^{ki}$ stands for the vector in $T_n(\mathbb{Z})$ such that $(\delta^{ki})_{rs}=\delta_{kr}\delta_{is}$.  If $\mathcal{C}$ is admissible, and $T(L)$ a $\mathcal{C}$-realization, we will call $V_{\mathcal{C}}(T(L))$ a  \emph{relation module}.
\end{definition}

%-----------------------------------------------------------%

\begin{remark}
\label{rem:eigenvalues}
By Equation \eqref{E_{k,k}}, whenever $\mathcal{C}$ is an admissible set of relations, $\mathcal{B}_{\mathcal{C}}(T(L))$ is an eigenbasis for the action of the Cartan subalgebra $\{E_{11},\ldots,E_{nn}\}$.  Moreover, any relation module is a Gelfand-Tsetlin module, with diagonalizable action of the Gelfand-Tsetlin subalgebra $\Gamma$ (see \cite[Theorem 5.3]{FRZ19}).
\end{remark}

%-----------------------------------------------------------%

	Set $k\neq n$, we say that $((k,i); (k,j))\in\mathfrak{V}\times \mathfrak{V}$  is an \emph{adjoining pair} for a graph $G(\mathcal{C})$, if  $i<j$,  $(k,i)\succeq_{\mathcal{C}}  (k,j)$, and $(k,i)\succeq_{\mathcal{C}} (k,t)\succeq_{\mathcal{C}} (k,j)$ implies $t=i$ or $t=j$.

%-----------------------------------------------------------%

\begin{theorem}[{\cite[Theorem 4.33]{FRZ19}}]
\label{thm:adm}
Suppose that $\mathcal{C}$ is a noncritical set of relations whose associated graph $G = G(\mathcal{C})$ satisfies the following conditions: 
	\begin{itemize}
	\item[(i)]  $G$ is reduced;
	\item[(ii)] $G$ does not contain loops, and $(k,i)\succeq_{\mathcal{C}}  (k,j)$ implies $i\leq j$;
	\item[(iii)] If  $G$ contains an arrow connecting $(k,i)$ and $(k+1,t)$, then  $(k+1,s)$ and $(k,j)$  with $i<j$, $s<t$ are not connected in $G$. 
	\end{itemize}
$\mathcal{C}$ is an admissible set of relations if and only if for any connected component $G(\mathcal{D})$ of $G(\mathcal{C})$ and any adjoining pair $((k,i);(k,j))$ in $G(\mathcal{D})$, there exist $p, q$ such that $\mathcal{E}_{1}\subseteq\mathcal{D}$ or, there exist $s<t$ such that $\mathcal{E}_{2}\subseteq\mathcal{D}$, where the graphs associated to $\mathcal{E}_{1}$ and $\mathcal{E}_{2}$ are as follows
	\begin{equation*}
	\begin{tabular}{c c c c}
	\xymatrixrowsep{0.3cm}
	\xymatrixcolsep{0.1cm}
	\xymatrix @C=0.2em{
	  &   &\scriptstyle{(k+1,p)}\ar@*{}[rd]   &   & \\
	 \scriptstyle{G(\mathcal{E}_{1})=}  &\scriptstyle{(k,i)}\ar@*{}[rd] \ar@*{}[ru]  &    &\scriptstyle{(k,j)};   &  \\
	  &   &\scriptstyle{(k-1,q)}\ar@*{}[ru]   &   & }
	&\ \ &
	\xymatrixrowsep{0.3cm}
	\xymatrixcolsep{0.1cm}
	\xymatrix @C=0.2em {
   	&   &\scriptstyle{(k+1,s)}    &   &\scriptstyle{(k+1,t)}\ar@*{}[rd]&& \\
  	\scriptstyle{G(\mathcal{E}_{2})=} &\scriptstyle{(k,i)} \ar@*{}[ru]  & &   & & \scriptstyle{(k,j)} \\
   	&   &   &   & &&}
	\end{tabular}
	\end{equation*}
\end{theorem}

%-----------------------------------------------------------%

\begin{remark}
Follows from Theorem~\ref{thm:adm} that the sets of relations $\mathcal{C}_{k}$, $\mathcal{C}^{+}_{k}$ and $\mathcal{C}^{-}_{k}$ from Example \ref{relac}, are admissible set of relations for any $1\leq k\leq n$. 
\end{remark}

%-----------------------------------------------------------%

	Gelfand-Tsetlin theorem \cite{GT50}, is one of the most remarkable results in representation theory and  gives an explicit realization of any simple finite dimensional module. The following theorem uses sets of relations to rewrite  Gelfand--Tsetlin theorem.

%-----------------------------------------------------------%

\begin{theorem}
\label{GT Theorem}
If $\lambda:=(\lambda_1,\ldots,\lambda_{n})$ is an integral dominant $\mathfrak{gl}_{n}$-weight and $T(\Lambda)$ is the Gelfand--Tsetlin tableau of height $n$ with entries $\lambda_{ki}:=\lambda_{i}$, then 
$V_{\mathcal{C}_{1}}(T(\Lambda))$ is isomorphic to the simple finite dimensional module $L(\lambda)$. Moreover,  
	\begin{itemize}
	\item[(i)] $\mathcal{B}_{\mathcal{C}_{1}}(T(\Lambda))$ is a basis of $V_{\mathcal{C}_{1}}(T(\Lambda))$.
	\item[(ii)] For any $\mu=(\mu_1,\ldots,\mu_n)\in\mathfrak{h}^{*}$, the weight space $L(\lambda)_{\mu}$ has a basis $$\left\{T(X)\in\mathcal{B}_{\mathcal{C}_{1}}(T(\Lambda))\ \Big | \ \mu_{k}=\sum_{i=1}^{k} x_{ki}-\sum_{i=1}^{k-1} x_{k-1,i}\text{ for all } k=1,\ldots,n\right\}.$$
	\end{itemize}

\end{theorem}

%-----------------------------------------------------------%

\begin{example}
\label{Some relation modules}
Let $\mathcal{C}$ be one of the sets of relations defined in Example \ref{relac}, $T(L)$ a $\mathcal{C}$-realization and $\lambda:=(l_{n1},\ldots,l_{nn})$. As we mention before, they are admissible sets of relations and:
	\begin{itemize}
	\item[(i)] Set $\mathcal{C}=\mathcal{C}_{1}^+$, in this case $V_{\mathcal{C}}(T(L))$ is isomorphic to the \emph{generic} Verma module $M(\lambda)$ (see \cite[Example 5.10]{FRZ19}). 
	
	\item[(ii)] Set $\mathcal{C}=\mathcal{C}_{2}$, and $\tilde{\lambda}:=(l_{n2},\ldots,l_{nn})$. In this case  $V_{\mathcal{C}}(T(L))$ is an infinite-dimensional module with finite-dimensional weight spaces (see \cite[Section 3]{Maz03}), all of them of dimension $\dim (L(\tilde{\lambda}))$. In fact, by Theorem \ref{GT Theorem}, $L(\tilde{\lambda})$ has a basis $\mathcal{B}$ parameterized by the set of standard tableaux in $T_{n-1}(\mathbb{Z})$ with top row $\tilde{\lambda}$. Given $\mu$ a weight of $V_{\mathcal{C}}(T(L))$, we consider the map $\psi_{\mu}:\mathcal{B}\to T_{n}(\mathbb{R})$ given by $\psi_{\mu}(T(\tilde{S}))=T(S)$ where:
	\begin{displaymath}
	{s}_{ij}=\begin{cases}
	\tilde{s}_{i-1,j-1}, & \text{if $i,j\geq 2$;}\\ 	
	\sum\limits_{t=1}^{i}\mu_{t}-\sum\limits_{r=2}^{i} \tilde{s}_{i-1,r-1}, &  \text{if $j=1$ and $2\leq i\leq n$;} \\
	\mu_{1}, & \text{if $(i,j)=(1,1)$.}
	\end{cases}
	\end{displaymath}
It is easy to check that $\psi_{\mu}$ is injective and $\psi(\mathcal{B})$ is a basis of $V_{\mathcal{C}}(T(L))_{\mu}$.
	
	\item[(iii)]  Set $\mathcal{C}=\mathcal{C}_{k}$ with $k\geq 3$. In this case, the module $V_{\mathcal{C}}(T(L))$ is infinite-dimensional with infinite-dimensional weight spaces. 
	\end{itemize}
\end{example}

%===========================================================%
		%POLYHEDRA ASSOCIATED WITH SETS OF RELATIONS%
%===========================================================%

\section{Polyhedra and its faces}
\label{sec:GRelat}

 We begin the section with some generalities about polyhedra and faces of polyhedra. Then we introduce the main objects to be studied in this paper, namely, polyhedra associated with sets of relations. We finish the section with some technical lemmas that will help us to characterize the dimension of a face associated with a given point.
 
%-----------------------------------------------------------%

Let $V$ be a finite-dimensional $\mathbb{R}$-vector space.  Given $v\in V$ and $W$ a vector subspace of $V$, the set  $v+W$ is called  \emph{affine subspace} of $V$. The dimension of $v+W$ is $\dim (W)$. For a subset $X$ of $V$ we will denote by $\textup{aff}(X)$ the  smallest affine subspace of $V$ containing $X$. 

%-----------------------------------------------------------%

A subset $P$ of an $\mathbb{R}$-vector space $V$ is called a \emph{polyhedron} if it is the intersection of finitely
many closed halfspaces. The \emph{dimension} of $P$ is given by $\dim \left(\textup{aff}(P)\right)$. A \emph{polytope} is a bounded polyhedron.

%-----------------------------------------------------------%

A hyperplane $H$ is called a \emph{support hyperplane} of the  polyhedron $P$, if $H\cap P\neq \emptyset$, and $P$ is contained in one of the two closed halfspaces bounded by $H$. The intersection $F=H\cap P$ is a \emph{face} of $P$, and $H$ is called a \emph{support hyperplane associated with} $F$. 
%A \emph{facet} of $P$ is a face of dimension $\dim P-1$. 
Faces of dimension $0$ are called \emph{vertices}, and an \emph{edge} is a face of dimension $1$.  In general, a face $F$ of dimension $k$ is called a \emph{$k$-face}.

%-----------------------------------------------------------%

\begin{proposition}[{\cite[Theorem 1.10(d)]{BG09}}]
\label{rem:int}
Let $P$ be a polyhedron in $\mathbb{R}^{d}$, and $x\in P$. There exists a unique face $F$ such that $x\in\textup{int}(F)$, where $\textup{int}(F)$ denotes the relative interior of the affine subspace $\textup{aff}(F)$ with respect to the standard topology of $\mathbb{R}^d$. 
\end{proposition}

%-----------------------------------------------------------%

	The face from Proposition \ref{rem:int} is the unique minimal  element in the set of faces of $P$ containing $x$ and is called \emph{minimal face} for $x$.

%-----------------------------------------------------------%

\begin{lemma}
\label{Lemma elements in the affine}
Let $P$ be a polyhedron of an $\mathbb{R}$-vector space $V$ and let $F$ be a face containing $x\in P$ with $\textup{aff}(F)=x+H$. If there exists $v$ such that $x\pm v\in P$, then $v\in H$.  
\end{lemma}

%-----------------------------------------------------------%

\begin{proof}
Suppose that $x-v,x+v\in P$, and $v\notin H$. Let us consider a support hyperplane $H_{\alpha}$ of $P$ such that $F=H_{\alpha}\cap P$ and $\alpha:V\longrightarrow\mathbb{R}$ its affine form, that is,
	$$
	\alpha(z)=\beta(z)+a_0,\ \ a_0=\alpha(0)
	$$
for some linear map $\beta:V\longrightarrow\mathbb{R}$. Since $H_{\alpha}$ is a support hyperplane of $P$, 
	$$
	P\subseteq H_{\alpha}^+:=\left\{z\in V\mid \alpha(z)\geq 0\right\},\ \ \ \mbox{or }\ \ P\subseteq H_{\alpha}^-:=\left\{z\in V\mid \alpha(z)\leq 0\right\}.
	$$
But $x-v\notin H_{\alpha}^+$, and $x+v\notin H_{\alpha}^-$, which is a contradiction.
\end{proof}

%-----------------------------------------------------------%

	Given a polyhedron $\mathcal{P}\subseteq \mathbb{R}^{\frac{n(n+1)}{2}}$ and $X\in \mathcal{P}$, by $D_{X}(\mathcal{P})\subseteq\mathbb{R}^{\frac{n(n+1)}{2}}$ we will denote the set $\left\{Y\in \mathbb{R}^{\frac{n(n+1)}{2}}\mid X+Y,X-Y\in \mathcal{P}\right\}$.

%-----------------------------------------------------------%

\begin{remark}
\label{Rem:dimFaceMin}
Let $\mathcal{P}$ be a polyhedron and $X\in \mathcal{P}$. If $\{Y_1,\ldots,Y_k\}$ is a maximal set of linearly independent vectors in $D_{X}(\mathcal{P})$, then the dimension of the minimal face in $\mathcal{P}$ containing $X$ is $k$. 
\end{remark}

%-----------------------------------------------------------%
		%POLYHEDRA ASSOCIATED WITH SETS OF RELATIONS%
%-----------------------------------------------------------%

\subsection{Polyhedra associated with sets of relations}

In this section, associated with any set of relations $\mathcal{C}$ we will define several polyhedra in  $\mathbb{R}^{\frac{n(n+1)}{2}}$, which will be the main objects of study until the end of this paper.

%-----------------------------------------------------------%

\begin{definition}
\label{ctab2}
Let $\mathcal{C}$ be any set of relations, $X\in\mathbb{R}^{\frac{n(n+1)}{2}}$ is called \emph{$\mathcal{C}$-pattern}, if $x_{ij}\geq x_{rs}$ for any $((i,j);(r,s))\in\mathcal{C}$. 
\end{definition}

%-----------------------------------------------------------%

\begin{remark}	
If $T(L)\in T_n(\mathbb{R})$ satisfies $\mathcal{C}$ (see Definition \ref{def:C-realization}), then $L$ is a $\mathcal{C}$-pattern. The converse is not necessarily true, in fact, let  $\mathcal{C}$ be the set of relations of Example \ref{example realizations} and $T(X)$ the tableau
	\begin{center}
	\Stone{\mbox{ \scriptsize {$3$}}}\Stone{\mbox{ \scriptsize {$3$}}}\Stone{\mbox{ \scriptsize {$5$}}}\Stone{\mbox{ \scriptsize {$\frac{7}{2}$}}}\\[0.2pt]
	\Stone{\mbox{ \scriptsize {$1$}}}\Stone{\mbox{ \scriptsize {$\frac{5}{2}$}}}\Stone{\mbox{ \scriptsize {$4$}}}\\[0.2pt]
 	\Stone{\mbox{ \scriptsize {$3$}}}\Stone{\mbox{ \scriptsize {$1$}}}\\[0.2pt]
	\Stone{\mbox{ \scriptsize {$\sqrt{2}$}}}\\
	\end{center}
in this case $X$ is a $\mathcal{C}$-pattern, and $T(X)$ that does not satisfies $\mathcal{C}$.
\end{remark}

%-----------------------------------------------------------%

For $1\leq k\leq n$ we define $R_{k}:\mathbb{C}^{\frac{n(n+1)}{2}}\longrightarrow\mathbb{C}$ given by $R_{k}(X):=\sum\limits_{i=1}^{k}x_{ki}$. The \emph{$k$th weight linear map}, $w_{k}:\mathbb{C}^{\frac{n(n+1)}{2}}\longrightarrow\mathbb{C}$ is defined by 
	\begin{displaymath}
	w_{k}(X):=
	\begin{cases}
	R_{k}(X)-R_{k-1}(X), & \text{if $2\leq k\leq n$;}\\
	x_{11}, &  \text{if $k=1$.}
	\end{cases}
	\end{displaymath}

%-----------------------------------------------------------%

\begin{definition}
\label{poliedro}
Fix $\lambda,\mu\in\R^{n}$ and $\mathcal{C}$ a set of relations. We consider the following polyhedra in $\mathbb{R}^{\frac{n(n+1)}{2}}$ associated with $\mathcal{C}$		
	\begin{align*}
	P_{\mathcal{C}}:=&\left\{X\in\mathbb{R}^{\frac{n(n+1)}{2}}\mid X \text{ is a }\mathcal{C}\text{-pattern}\right\},\\
	P_{\mathcal{C}}(\lambda):=&\left\{X\in P_{\mathcal{C}}\mid x_{nj}=\lambda_j\text{ for all }1\leq j\leq n\right\},\\
	P_{\mathcal{C}}(\lambda,\mu):=&\left\{X\in P_{\mathcal{C}}(\lambda)\mid w_{i}(X)=\mu_{i}\text{ for all }1 \leq i\leq n\right\}.
	\end{align*}

\end{definition}

%-----------------------------------------------------------%

	Given $y\in\mathbb{R}$, the point $Y$ with entries $y_{ij}=y$ is a $\mathcal{C}$-pattern for any $\mathcal{C}$, in particular  $P_{\mathcal{C}}$ is always unbounded. However, $P_{\mathcal{C}}(\lambda)$ is a polytope if and only if the maximal and minimal points with respect to $\succeq_{\mathcal{C}}$, belong to the set $\{(n,1),\ldots,(n,n)\}$.
\begin{remark}	
	If $\lambda\in \mathbb{Z}^{n}$ is an integral dominant $\mathfrak{gl}_{n}$-weight with $\lambda_n\geq 0$ and $\mu \in \mathbb{Z}^{n}$ a weight of $L(\lambda)$, then $P_{\mathcal{C}_1}(\lambda,\mu)$ is a polytope called \emph{Gelfand--Tsetlin polytope} associated with $\lambda$ and $\mu$. 
\end{remark}

%-----------------------------------------------------------%

\begin{lemma}
\label{somazeroger}
Let $X$ be a $\mathcal{C}$-pattern in $P_{\mathcal{C}}(\lambda,\mu)$, and $Y$ a  $\mathcal{C}$-pattern.
	\begin{itemize}
	\item[(i)] $X+Y\in P_{\mathcal{C}}(\lambda)$ if and only if $y_{ni}=0$ for any $1\leq i\leq n$.
	\item[(ii)] $X+Y\in P_{\mathcal{C}}(\lambda,\mu)$ if and only if $X+Y\in P_{\mathcal{C}}(\lambda)$, and $R_{k}(Y)=0$ for any $1\leq k\leq n$.
	\end{itemize}
\end{lemma}

%-----------------------------------------------------------%

\begin{proof}
As the tableau associated with the sum of  $\mathcal{C}$-patterns is a $\mathcal{C}$-pattern, in order  to prove (i) we just note that $x_{ni}=\lambda_{i}=x_{ni}+y_{ni}$ if and only if $y_{ni}=0$. To prove (ii) we note that $X+Y\in P_{\mathcal{C}}(\lambda)$ belongs to $P_{\mathcal{C}}(\lambda,\mu)$ if and only if $\sum\limits_{i=1}^{k}\mu_{i}=R_{k}(X+Y)$, and $ R_{k}(X+Y)=R_{k}(X)+R_{k}(Y)=\sum\limits_{i=1}^{k}\mu_{i}+R_{k}(Y)$ for any $1\leq k\leq n$.
\end{proof}

%-----------------------------------------------------------% 

	The following results establish a direct connection between relation modules and the polyhedra $P_{\mathcal{C}}$, $P_{\mathcal{C}}(\lambda)$, and $P_{\mathcal{C}}(\lambda,\mu)$ from Definition \ref{poliedro}. Recall that tableaux in $T_{n}(L+{\mathbb Z}_0^\frac{n(n+1)}{2})$ are called $L$-integral. We say that a point $X$ is \emph{$L$-integral} if $T(X)$ is an $L$-integral tableau.

%-----------------------------------------------------------% 

\begin{theorem}
\label{number of L-integral tableaux}
Let $\mathcal{C}$ be any admissible set of relations, $T(L)$ a $\mathcal{C}$-realization, and $V=V_{\mathcal{C}}(T(L))$ the corresponding relation $\mathfrak{gl}_n$-module. Set $\lambda=(l_{n1},\ldots,l_{nn})$, and $\mu=(w_{1}(L),w_{2}(L),\ldots,w_{n}(L))\in\mathfrak{h}^{*}$. 
	\begin{itemize}
	\item[(i)] The polyhedra $P_{\mathcal{C}}$ and $P_{\mathcal{C}}(\lambda)$ have the same number of $L$-integral points, and this number is equal to $\dim(V)$.
	\item[(ii)] The number of $L$-integral points in $P_{\mathcal{C}}(\lambda,\mu)$ is equal to $\dim(V_{\mu})$.
	\end{itemize}
\end{theorem}

%-----------------------------------------------------------%

\begin{proof}
 First we note that by definition the sets of $L$-integral points of $P_{\mathcal{C}}$ and $P_{\mathcal{C}}(\lambda)$ coincide. By construction, $T_n\left(P_{\mathcal{C}}(\lambda)\right)\cap T_{n}(L+\mathbb{Z}_0^{\frac{n(n+1)}{2}})$ is a basis of the module $V_{\mathcal{C}}(T(L))$, which implies (i). Analogously, $T_n\left(P_{\mathcal{C}}(\lambda,\mu)\right)\cap T_{n}(L+\mathbb{Z}_0^{\frac{n(n+1)}{2}})$ is a basis of the weight space $V_{\mu}$, implying (ii). 
\end{proof}

%-----------------------------------------------------------%

	As a consequence of Theorem \ref{number of L-integral tableaux} we have the following:

%-----------------------------------------------------------%

\begin{corollary} 
Let $\mathcal{C}$ be an admissible set of relations, $T(L)$ a $\mathcal{C}$-realization, $\lambda=(l_{n1},\ldots,l_{nn})$, and $\mu$ a weight of $V_{\mathcal{C}}(T(L))$. 
	\begin{itemize}
	\item[(i)]  Set $\mathcal{C}=\mathcal{C}_{1}$. In this case, the module $V_{\mathcal{C}}(T(L))$ is isomorphic to the simple finite dimensional module $L(\lambda)$. 
		\begin{itemize}
		\item The number of $L$-integral points in $P_{\mathcal{C}}(\lambda)$ is finite and equal to the dimension of $L(\lambda)$.
		\item The number of $L$-integral points  in $P_{\mathcal{C}}(\lambda,\mu)$ is finite and equal to the dimension of the weight space  $L(\lambda)_{\mu}$.
		\end{itemize}

	\item[(ii)] Set $\mathcal{C}=\mathcal{C}_{1}^{+}$. In this case, the module $V_{\mathcal{C}}(T(L))$ is isomorphic to the generic Verma module $M(\lambda)$ (see \cite[Example 5.10]{FRZ19}), and
		\begin{itemize}
		\item $P_{\mathcal{C}}(\lambda)$ contains infinitely many $L$-integral  points.
		\item The number of $L$-integral points in $P_{\mathcal{C}}(\lambda,\mu)$ is $\dim (M(\lambda)_{\mu})<\infty$.
		\end{itemize}

	\item[(iii)]  Set $\mathcal{C}=\mathcal{C}_{2}$. In this case, $\tilde{\lambda}=(l_{n2},\ldots,l_{nn})$ is a dominant $\mathfrak{gl}_{n-1}$-weight, and $V_{\mathcal{C}}(T(L))$ is an infinite-dimensional module with finite weight spaces of dimension $\dim(L(\tilde{\lambda}))$ (see Example \ref{Some relation modules}).
		\begin{itemize}
		\item $P_{\mathcal{C}}(\lambda)$ contains infinitely many $L$-integral points.
		\item If $\mu$ is a weight of $V_{\mathcal{C}}(T(L))$, the number of $L$-integral points in $P_{\mathcal{C}}(\lambda,\mu)$ is equal to $\dim (L(\tilde{\lambda}))$.
		\end{itemize}
	\end{itemize}

\end{corollary}

%-----------------------------------------------------------%
					%FACES OF POLYHEDRA%
%-----------------------------------------------------------%

\begin{definition}
\label{mosaico}
Let $\mathcal{C}$ be a set of relations, and $X\in\mathbb{R}^{\frac{n(n+1)}{2}}$ be a $\mathcal{C}$-pattern. Associated with $\mathcal{C}$ and $T(X)$  we have an equivalence relation in $\mathfrak{V}$, where $(i,j)$ is related with $(r,s)$ if and only if there exists a walk in $G(\mathcal{C})$ connecting $(i,j)$ and $(r,s)$ with the entries of $X$ associated with the vertices in the walk being equal. The partition of $\mathfrak{V}$ induzed by the relation is called \emph{tiling}, and is denoted  by $\mathcal{M}_{\mathcal{C}}(X)$.  The equivalence classes will be called \emph{tiles}. In particular, if $(i,j)\notin\mathfrak{V}(\mathcal{C})$, the set $\{(i,j)\}$ is a tile.   
\end{definition}

%-----------------------------------------------------------%

	Set $\Lambda\subseteq \{1,n\}$.  A tile $\mathcal{M}$ will be called \emph{$\Lambda$-free} if $\mathcal{M}\cap \{(k,j)\  |\ 1\leq j\leq k\}=\emptyset$ for any $k\in \Lambda$. Moreover, if $(i,j)$ belongs to a $\Lambda$-free tile, it  will be called \emph{$\Lambda$-free}. Note that any tile is $\emptyset$-free tile.

%-----------------------------------------------------------%

\begin{remark}
In this paper we will be interested in the particular cases $\Lambda_1:=\{n\}$, and $\Lambda_2:=\{1,n\}$. It is worth to mention that in \cite{LM04} the authors call free tiles what we will call here $\Lambda_2$-free tiles.
\end{remark}
 
%-----------------------------------------------------------%

	Figure $\ref{figctab}$ shows a tableau in $T_{5}(\mathbb{R})$, and its corresponding tiling. The tiles without colouring in the right hand side are the $\Lambda_2$-free tiles.

%-----------------------------------------------------------%

\begin{figure}
\begin{minipage}[l]{6cm}
%\label{figctab}
	\begin{center}
	\Stone{\mbox{ \scriptsize {$9$}}}\Stone{\mbox{ \scriptsize {$8$}}}\Stone{\mbox{ \scriptsize {$6$}}}\Stone{\mbox{ \scriptsize {$5$}}}\Stone{\mbox{ \scriptsize {$3$}}}\\[0.2pt]
	\Stone{\mbox{ \scriptsize {$8$}}}\Stone{\mbox{ \scriptsize {$5$}}}\Stone{\mbox{ \scriptsize {$5$}}}\Stone{\mbox{ \scriptsize {$4$}}}\\[0.2pt]
	\Stone{\mbox{ \scriptsize {$3$}}}\Stone{\mbox{ \scriptsize {$3$}}}\Stone{\mbox{ \scriptsize {$0$}}}\\[0.2pt]
 	\Stone{\mbox{ \scriptsize {$3$}}}\Stone{\mbox{ \scriptsize {$-1$}}}\\[0.2pt]
	\Stone{\mbox{ \scriptsize {$-2$}}}\\
\end{center}

\end{minipage}
\begin{minipage}[r]{6cm}
\includegraphics[width=6cm]{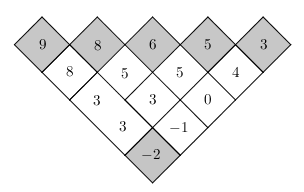}
\end{minipage}
\caption{Tiling of a $\mathcal{C}_1^{+}$-pattern}\label{figctab}
\end{figure}

%-----------------------------------------------------------%

\begin{lemma}
\label{same entries in the tiles}
Let $\mathcal{C}$ be a set of relations, $X$ a $\mathcal{C}$-pattern, $\lambda=(x_{n1},\ldots,x_{nn})$, and $\mu=(w_{1}(X),w_{2}(X),\ldots,w_{n}(X))$. 
  \begin{itemize}
 \item[(i)] If $Y\in D_{X}(\mathcal{P}_{\mathcal{C}})$, then $y_{ij}=y_{rs}$, whenever $(i,j)$ and $(r,s)$ are in the same  tile of $\mathcal{M}_{\mathcal{C}}(X)$. 
  \item[(ii)]   If $Y\in D_{X}(\mathcal{P}_{\mathcal{C}}(\lambda))$, then $y_{ij}=0$, whenever $(i,j)$ is not $\Lambda_1$-free.
 \item[(iii)]  If $Y\in D_{X}(\mathcal{P}_{\mathcal{C}}(\lambda,\mu))$, then $y_{ij}=0$, whenever $(i,j)$ is not $\Lambda_2$-free.   

\end{itemize}

\end{lemma}

%-----------------------------------------------------------%

\begin{proof}
If $(i,j)=(r,s)$ there is nothing to prove, so we can assume $(i,j)$ and $(r,s)$ are in the same tile and $(i,j)\neq(r,s)$. Consider a walk in $G(\mathcal{C})$ connecting $(i,j)$ and $(r,s)$ with the entries associated with the vertices in the walk being equal.  Arguing by induction on the length of the walk, it is enough to prove the lemma for walks of length $1$. Without loss of generality we can assume that such a walk is given by the relation $((i,j);(r,s))\in\mathcal{C}$. In this case, the entries of any $\mathcal{C}$-pattern $Z$ should satisfy $z_{ij}\geq z_{rs}$. Now, by the hypothesis 
	$$
	x_{ij}+y_{ij}\geq x_{rs}+y_{rs},
	\ \ \
	x_{ij}-y_{ij}\geq x_{rs}-y_{rs}, \ \ \mbox{and }\
	x_{ij}= x_{rs}
	$$
which implies (i). 

	By Lemma \ref{somazeroger}(i) we have $y_{ni}=0$ for $1\leq i\leq n$, therefore (ii) is a consequence of item (i) and $\mathcal{P}_{\mathcal{C}}(\lambda)\subseteq \mathcal{P}_{\mathcal{C}}$. Moreover, under the assumptions of item (iii), Lemma \ref{somazeroger}(ii) implies also $y_{11}=0$, now (iii) follows from the definition of $\Lambda_2$-free tile and item (i). 
\end{proof}

%-----------------------------------------------------------%

\begin{lemma}
\label{No free tiles}
Let $\mathcal{C}$ be a set of relations, $X$ a $\mathcal{C}$-pattern, $\lambda=(x_{n1},\ldots,x_{nn})$, and $\mu=(w_{1}(X),w_{2}(X),\ldots,w_{n}(X))$. 

	\begin{itemize}
	\item[(i)] If $\mathcal{M}_{\mathcal{C}}(X)$ does not have $\Lambda_1$-free tiles, then $X$ is a vertex of $P_{\mathcal{C}}(\lambda)$.
	\item[(ii)] If $\mathcal{M}_{\mathcal{C}}(X)$ does not have $\Lambda_2$-free tiles, then $X$ is a vertex of $P_{\mathcal{C}}(\lambda,\mu)$.
	\end{itemize}
\end{lemma}

%-----------------------------------------------------------%

\begin{proof}
	To show that $X$ is a vertex of a polyhedron $\mathcal{P}$ it is enough to proof that $D_{X}(\mathcal{P})=\{0\}$. With this in mind, (i) follows from Lemma \ref{same entries in the tiles}(ii), and (ii) follows from Lemma \ref{same entries in the tiles}(iii).
\end{proof}

%-----------------------------------------------------------%

\begin{definition}
\label{matmosa}
Given a $\mathcal{C}$-pattern $X$, with tiling $\mathcal{M}_{\mathcal{C}}(X)$ and set of $\Lambda_{1}$-free tiles $\mathcal{M}_{1},\mathcal{M}_{2},\ldots,\mathcal{M}_{s}$, we define a \emph{tiling matrix}  $A_{\mathcal{M}_{\mathcal{C}}(X)}$  to be the matrix 
 	$$
	\left(a_{ik}\right)_{\substack{
    1\leq i\leq n-1\\
    1\leq k\leq s}}\ \ \mbox{where }\ a_{ik}=\left|\left\{j\mid\left(i,j\right)\in \mathcal{M}_{k}\right\}\right|.
    $$
 For $\mathcal{C}$-patterns without  $\Lambda_1$-free tiles we consider $A_{\mathcal{M}_{\mathcal{C}}(X)}$ to be the identity matrix of order $n-1$. 
 
\end{definition}

%-----------------------------------------------------------%

\begin{remark}
\label{order tiles}
Note that the matrix $A_{\mathcal{M}_{\mathcal{C}}(X)}$ depends of the chosen order of the $\Lambda_1$-free tiles, but the dimension of the kernel of $A_{\mathcal{M}_{\mathcal{C}}(X)}$ does not depend. In fact, given two different orders of the $\Lambda_1$-free tiles, the corresponding matrices are related by a permutation of the columns. 
\end{remark}

%-----------------------------------------------------------%

\begin{example}
The  tiling matrix associated to the tableau in Figure $\ref{figctab}$, where the tiles are enumerated from left to right and from bottom to top, is: 
	\begin{align*}
	A_{\mathcal{M}_{\mathcal{C}}(X)} &
	=\left(\begin{array}{ccccccccc}
	1&0&0&0&0&0&0&0&0\\
	0&1&1&0&0&0&0&0&0\\
	0&1&0&1&1&0&0&0&0\\
	0&0&0&0&0&1&1&1&1
	\end{array}\right)
	\end{align*}
 
\end{example}

%-----------------------------------------------------------%

\section{Main results}
\label{sec:main}
	From now on and until the end of this paper we will fix a set of relations $\mathcal{C}$, a $\mathcal{C}$-pattern $X$, $\lambda=(x_{n1},\ldots,x_{nn})$, and $\mu=(w_{1}(X),w_{2}(X),\ldots,w_{n}(X))$. By $d$ we will denote the number of tiles, by $s$ the number of $\Lambda_1$-free tiles, and $r$ the dimension of  $\ker \left(A_{\mathcal{M}_{\mathcal{C}}(X)}\right)\subseteq\mathbb{R}^s$, so $r\leq s\leq d$.  Whenever $r\neq 0$, we fix a basis $\left\lbrace \widehat{\varepsilon}^{(1)},\widehat{\varepsilon}^{(2)},\ldots,\widehat{\varepsilon}^{(r)}\right\rbrace$ of $\ker \left(A_{\mathcal{M}_{\mathcal{C}}(X)}\right)\subseteq\mathbb{R}^s$. By Remark \ref{order tiles}, we will assume that $\mathcal{M}_{\mathcal{C}}(X)=\{\mathcal{M}_1,\ldots,\mathcal{M}_s,\ldots,\mathcal{M}_d\}$, where $\{\mathcal{M}_1,\ldots,\mathcal{M}_{s}\}$ is the set of all $\Lambda_1$-free tiles.
	
%-----------------------------------------------------------%

\begin{notation}
\label{nota:main}
Let us consider a basis $\left\lbrace \varepsilon^{(1)},\ldots,\varepsilon^{(d)}\right\rbrace$ of $\mathbb{R}^d$ with the following properties:  If $\ker \left(A_{\mathcal{M}_{\mathcal{C}}(X)}\right)=\{0\}$,  $\varepsilon^{(i)}$ will be the $i$th canonical vector of  $\mathbb{R}^d$. In the case $\ker \left(A_{\mathcal{M}_{\mathcal{C}}(X)}\right)\neq \{0\}$, for $i\geq r+1$,  $\varepsilon^{(i)}$ will be the $i$th canonical vector of  $\mathbb{R}^d$,  and  $\left\lbrace \varepsilon^{(1)},\varepsilon^{(2)},\ldots,\varepsilon^{(r)}\right\rbrace$ are defined by
	\begin{displaymath}
	\varepsilon_{k}^{(m)}= 
	\begin{cases}
		\widehat{\varepsilon}_{k}^{(m)}, & \text{if $k\leq s$;}\\
		0, & \text{if $k> s$,}
			\end{cases}
	\end{displaymath}
where $\textup{\textbf{x}}_{k}$ denotes the $k$th coordinate of a vector $\textup{\textbf{x}}$. Now, consider the linear map $\psi: \mathbb{R}^d  \longrightarrow  \mathbb{R}^{\frac{n(n+1)}{2}}$  defined by  $\psi(\varepsilon^{(m)})=Y^{(m)}$, where $y^{(m)}_{ij}=\varepsilon^{(m)}_{k}$  if $(i,j)\in \mathcal{M}_{k}$.

 	Finally, denote by $\mathcal{Y}_t$ the set $\left\lbrace Y^{(1)},\ldots,Y^{(t)}\right\rbrace$, for $t=1,2,\dots,d$. By convention, $\mathcal{Y}_0=\{0\}$. 

\end{notation}

%-----------------------------------------------------------%

\begin{lemma}
\label{lem:Dx(P)}
Under the previous notation. 
	\begin{itemize}
	\item[(i)] If $X\in\mathcal{P}_{\mathcal{C}}$, then $\mathcal{Y}_d\subseteq D_X\left(\mathcal{P}_{\mathcal{C}}\right)$.
	\item[(ii)] If $X\in\mathcal{P}_{\mathcal{C}}(\lambda)$, then $\mathcal{Y}_s\subseteq D_X\left(\mathcal{P}_{\mathcal{C}}\left(\lambda\right)\right)$.
	\item[(iii)] If $X\in\mathcal{P}_{\mathcal{C}}(\lambda,\mu)$, then $\mathcal{Y}_r\subseteq D_X\left(\mathcal{P}_{\mathcal{C}}\left(\lambda,\mu\right)\right)$.
\end{itemize}
\end{lemma}

%-----------------------------------------------------------%

\begin{proof}
To prove (i) we show that $X\pm Y^{(m)}\in \mathcal{P}_{\mathcal{C}}$ for $1\leq m\leq d$. As $X\pm Y\in P_{\emptyset}$ for any $Y\in \mathbb{R}^{\frac{n(n+1)}{2}}$, we can assume without lose of generality $\mathcal{C}\neq \emptyset $. Let us consider a relation $((i,j);(r,s))\in\mathcal{C}$. Suppose first that  $x_{ij}=x_{rs}$, in this case $(i,j)$ and $(r,s)$ belong to the same tile $\mathcal{M}_k$, and therefore  $y_{ij}^{(m)}=\varepsilon_k^{(m)}=y_{rs}^{(m)}$, which implies $\left(x_{ij}\pm y^{(m)}_{ij}\right)-\left(x_{rs}\pm y^{(m)}_{rs}\right)=0$. Suppose now that $x_{ij}> x_{rs}$ and assume   $$\left| \varepsilon_{k}^{(m)}\right| < \min\left\{ \frac{\left| x_{rs}-x_{pq}\right|}{2}\colon x_{rs}\neq x_{pq}\right\}$$ for $1\leq m,k\leq d$. Set $\epsilon^{(m)}:=\max\left\{\left|\varepsilon_k^{(m)}\right| \colon 1\leq k\leq d\right\}$. In this case
	\begin{align*}
	\pm\left(y_{ij}^{(m)}-y_{rs}^{(m)}\right)&\leq \left| y_{ij}^{(m)}-y_{rs}^{(m)}\right| \leq \left| y_{ij}^{(m)}\right| + \left| y_{rs}^{(m)}\right| \leq 2\epsilon^{(m)}\\ 	
	&< \min\left\{ \left| x_{rs}-x_{pq}\right|\colon x_{rs}\neq x_{pq}\right\}\\
	&\leq \left| x_{ij}-x_{rs} \right| = x_{ij}-x_{rs}.
	\end{align*} 
This implies $x_{ij}-x_{rs}\pm\left(y_{ij}^{(m)}-y_{rs}^{(m)}\right)\geq 0$.

	To prove (ii) we show that $X\pm Y^{(m)}\in P_{\mathcal{C}}(\lambda)$ for all $1\leq m\leq s$, by item (i) we have $Y^{(m)}\in D_X\left(\mathcal{P}_{\mathcal{C}}\right)$ for all $m=1,2,\dots,s$, and by Lemma \ref{somazeroger}(i), it is enough to prove that $y^{(m)}_{nj}=0$ for $1\leq j\leq n$, which follows from the fact that $\left(n,j\right)$ do not belong to any $\Lambda_1$-free tile. 

	Let us show that $X\pm Y^{(m)}\in P_{\mathcal{C}}(\lambda,\mu)$ for all $1\leq m\leq r$. For any $1\leq i\leq n-1$ we have
	$$
	R_i\left(Y^{(m)}\right)=\sum_{j=1}^{i}y_{ij}^{(m)}=\sum_{k=1}^{s}a_{ik}\varepsilon_{k}^{(m)},
	$$
where $a_{ik}=\left|\left\{j\mid\left(i,j\right)\in \mathcal{M}_{k}\right\}\right|$. The right hand side of the previous equality is the dot product between $\varepsilon^{(m)}\in \ker\left(A_{\mathcal{M}_{\mathcal{C}}(X)}\right)$ and the $i$th row of $A_{\mathcal{M}_{\mathcal{C}}(X)}$, therefore equal to zero. It follows from item (ii) and  Lemma \ref{somazeroger}(ii) that $X\pm Y^{(m)} \in P_{\mathcal{C}}\left(\lambda,\mu\right)$, which completes the proof of (iii).
\end{proof}

%-----------------------------------------------------------%

\begin{definition}
Let $\mathcal{C}$ be a set of relations, and $\mathcal{P}\subseteq \mathcal{P}_{\mathcal{C}}$ be any polyhedron. Associated with $\mathcal{P}$ and $\mathcal{C}$  we will consider the following polyhedron: 
	$$
	\mathcal{P}^{+}:=\{X\in \mathcal{P}\mid x_{ij}\geq 0 \text{ whenever } (i,j)\in\mathfrak{V}(\mathcal{C})\}.
	$$
$\mathcal{C}$ will be called \emph{top-connected}, if for each $(i,j)\in\mathfrak{V}(\mathcal{C})$ with $i\neq n$, there exists $r$  such that $(i,j)\succeq_{\mathcal{C}}  (n,r)$. 
\end{definition}

%-----------------------------------------------------------%

\begin{remark}\label{++}
	In the classical definition of Gelfand--Tsetlin pattern (cf. \cite{BZ89, KB95, LM04, ABS11, LMD19})  it is also required for the entries of the tableau to be non-negative. More concretely, the term Gelfand--Tsetlin pattern is  used  for elements in $\mathcal{P}_{\mathcal{C}_1}^{+}$. In particular, Gelfand-Tsetlin polytopes are defined as $\mathcal{P}_{\mathcal{C}_1}^{+}(\lambda,\mu)$. It became clear the necessity of discussing which properties of a polyhedra $\mathcal{P}$ are also satisfied  by $\mathcal{P}^{+}$.
	\end{remark}
	
%-----------------------------------------------------------%

If $\lambda$ is an integral dominant weight with non-negative entries we have $\mathcal{P}_{\mathcal{C}_1}(\lambda,\mu)=\mathcal{P}^{+}_{\mathcal{C}_1}(\lambda,\mu)$. Note also that $\mathcal{C}_{k}$ and $\mathcal{C}^{-}_{k}$ from Example \ref{relac} are top-connected for any $1\leq k\leq n$, however $\mathcal{C}^{+}_{k}$ is top-connected only if $k=n$.

%-----------------------------------------------------------%

\begin{lemma}\label{positive}
Let $\mathcal{C}$ be a top-connected set of relations. 
	\begin{itemize}
	\item[(i)]If $X\in \mathcal{P}^{+}_{\mathcal{C}}(\lambda)$, then $\mathcal{Y}_s\subseteq D_X\left(\mathcal{P}^{+}_{\mathcal{C}}(\lambda)\right)$.
	\item[(ii)] If $X\in \mathcal{P}^{+}_{\mathcal{C}}(\lambda,\mu)$, then $\mathcal{Y}_r\subseteq D_X\left(\mathcal{P}^{+}_{\mathcal{C}}(\lambda,\mu)\right)$. 
	\end{itemize}
\end{lemma}

%-----------------------------------------------------------%

\begin{proof}
	We only prove (i), the proof of item (ii) is analogue using Lemma \ref{same entries in the tiles}(iii). By Lemma \ref{lem:Dx(P)}(ii) we have $\mathcal{Y}_s\subseteq D_X\left(\mathcal{P}_{\mathcal{C}}(\lambda)\right)$, so it is enough to prove that $x_{ij}\pm y_{ij}^{(m)}\geq 0$ whenever  $(i,j)\in \mathfrak{V}(\mathcal{C})$ and $1\leq m\leq s$. If $(i,j)$ is not  $\Lambda_{1}$-free, then by Lemma \ref{same entries in the tiles}(ii) we have  $y_{ij}^{(m)}=0$ and therefore $x_{ij}\pm y_{ij}^{(m)}=x_{ij}\geq 0$. On the other hand, suppose $(i,j)$ belongs to some  $\Lambda_{1}$-free tile $\mathcal{M}_k$, that  implies $i\neq n$. As $\mathcal{C}$ is  top-connected, there exists $1\leq r\leq n$ such that  $(i,j)\succeq_{\mathcal{C}}  (n,r)$, so $(n,r)\in\mathfrak{V}(\mathcal{C})$, and  $x_{ij}\geq x_{nr}\geq 0$. Moreover $x_{ij}\neq x_{nr}$ because $(i,j)$ and $(n,r)$ do not belong to the same tile. Rescaling $Y^{(m)}$ if necessary, we can assume $$\left| \varepsilon_{k}^{(m)}\right| < \min\left\{ \frac{\left| x_{rs}-x_{pq}\right|}{2}\colon x_{rs}\neq x_{pq}\right\}$$ for $1\leq m,k\leq d$, so we have
	\begin{equation*}
	\begin{split}
	\pm y_{ij}^{(m)}&=\pm\varepsilon^{(m)}_{k}\leq \left|\varepsilon^{(m)}_{k}\right| < \frac{1}{2}\min\left\{\left| x_{rs}-x_{pq}\right|\colon x_{rs}\neq x_{pq}\right\} \\
	&\leq \left| x_{ij}-x_{nr}\right| = x_{ij}-x_{nr}\leq x_{ij}.
	\end{split}
	\end{equation*}
Hence $x_{ij}\pm y_{ij}^{(m)}\geq 0$ for all $(i,j)\in\mathfrak{V}(\mathcal{C})$. 
\end{proof}

%-----------------------------------------------------------%

\begin{remark}\label{contraexemplo PC}
Note that in general it is not true that $\mathcal{Y}_d\subseteq D_X\left(\mathcal{P}^{+}_{\mathcal{C}}\right)$. For instance, consider $X$ to be the tableau with all entries being $0$, and $\mathcal{C}$ any set of relations with $\mathfrak{V}(\mathcal{C})=\mathfrak{V}$. It is trivial to show that $D_X\left(\mathcal{P}^{+}_{\mathcal{C}}\right)=\{0\}$. 
\end{remark}

%-----------------------------------------------------------%

	The following theorem generalizes \cite[Theorem 1.5]{LM04}. 

%-----------------------------------------------------------%

\begin{theorem}
\label{thm:dimI}
	Let $\mathcal{C}$ be any set of relations, $X$ a $\mathcal{C}$-pattern, and $\mathcal{M}_{\mathcal{C}}(X)$ its  associated tiling. Set $\lambda=(x_{n1},\ldots,x_{nn})$, and $\mu=(w_{1}(X),w_{2}(X),\ldots,w_{n}(X))$. Then
	\begin{itemize}
	\item[(i)] The dimension of the minimal face of $P_{\mathcal{C}}$ containing $X$ is equal to the number of tiles in $\mathcal{M}_{\mathcal{C}}(X)$.
	\item[(ii)] The dimension of the minimal face of $P_{\mathcal{C}}(\lambda)$  containing $X$ is equal to the number of $\Lambda_1$-free tiles in $\mathcal{M}_{\mathcal{C}}(X)$.
		\item[(iii)]  The dimension of the minimal face of $P_{\mathcal{C}}(\lambda,\mu)$  containing $X$ is equal to the dimension of the kernel of $A_{\mathcal{M}_{\mathcal{C}}(X)}$.
	\end{itemize}
\end{theorem} 

%-----------------------------------------------------------%

\begin{proof}
	Let $H_{\lambda\mu},H_{\lambda}$ and $H$ be the $\mathbb{R}$-vector subspaces of  $\mathbb{R}^{\frac{n(n+1)}{2}}$ such that $H_{\lambda\mu}+X$, $H_{\lambda}+X$, and $H+X$ are the affine span of the minimal face of $P_{\mathcal{C}}(\lambda,\mu)$,  $P_{\mathcal{C}}(\lambda)$ and $P_{\mathcal{C}}$ containing $X$, respectively.  Let us consider the bases and linear maps as mentioned in Notation \ref{nota:main}. 

	Since $\varepsilon^{(1)},\ldots,\varepsilon^{(d)}$ are linearly independent, we conclude that $Y^{(1)},\ldots,Y^{(d)}$  are linearly  independent. Due to Lemmas \ref{Lemma elements in the affine} and \ref{lem:Dx(P)} we also have $\mathcal{Y}_r\subseteq H_{\lambda\mu}$, $\mathcal{Y}_s\subseteq H_{\lambda}$ and $\mathcal{Y}_d\subseteq H$. We finish the proof showing that $H_{\lambda\mu}$ is spanned by $\mathcal{Y}_r$, $H_{\lambda}$ is spanned by $\mathcal{Y}_s$, and $H$ is spanned by $\mathcal{Y}_d$ . 

	Set $Y\in H$ such that $X\pm Y\in P_{\mathcal{C}}$, and consider $\varepsilon:=(\varepsilon_1,\ldots,\varepsilon_d)\in\mathbb{R}^d$, where $\varepsilon_{k}:=y_{ij}$ whenever $\left(i,j\right)\in\mathcal{M}_{k}$ (note that $\varepsilon_{k}$ is well-defined by Lemma~\ref{same entries in the tiles} (i)). By construction of $\psi$, we have $\psi\left(\varepsilon\right)=Y$, and therefore $Y\in\textup{span}\ \mathcal{Y}_d$, because $\left\{\varepsilon^{(1)},\ldots,\varepsilon^{(d)}\right\}$ is a basis of $\mathbb{R}^d$.
	
 If $\mathcal{M}_{\mathcal{C}}(X)$ does not have $\Lambda_1$-free tiles, items (ii) and (iii) are consequence of Lemma \ref{No free tiles}.  Suppose that $\mathcal{M}_{\mathcal{C}}(X)$ has at least one $\Lambda_1$-free tile.

If $Y\in H_{\lambda}$ is such that $X\pm Y\in P_{\mathcal{C}}\left(\lambda\right)$, Lemma~\ref{same entries in the tiles} (ii) implies that $y_{ij}=0$ if $\left(i,j\right)$ is not $\Lambda_1$-free. As  $H_{\lambda}\subseteq H$ and $P_{\mathcal{C}}\left(\lambda\right)\subseteq P_{\mathcal{C}}$ we have $\varepsilon=(\varepsilon_1,\ldots,\varepsilon_s,0,0,\dots,0)$. Finally, as $\psi\left(\varepsilon\right)=Y$, we have $Y\in\textup{span}\ \mathcal{Y}_s$, by the construction of $\psi$ and the basis $\left\{\varepsilon^{(1)},\ldots,\varepsilon^{(s)}\right\}$ of $\mathbb{R}^s$.

 Consider now  $Y\in H_{\lambda\mu}$ such that $X\pm Y\in P_{\mathcal{C}}\left(\lambda,\mu\right)$. Lemma~\ref{same entries in the tiles} (iii) implies that $y_{ij}=0$ if $\left(i,j\right)$ is not $\Lambda_2$-free,  hence $\varepsilon_i=0$ for $i\in\{1,s+1,s+2,\ldots,d\}$. Consider $\widehat{\varepsilon}:=(\varepsilon_1,\ldots,\varepsilon_s)$. By the conditions on $Y$, we  get $\widehat{\varepsilon}\in \ker\left( A_{\mathcal{M}_{\mathcal{C}}(X)}\right)$, in fact, from Lemma \ref{somazeroger} (ii) we have $\sum\limits_{k=1}^{s}a_{ik}\varepsilon_{k}=\sum\limits_{j=1}^{i}y_{ij}=0$, where $a_{ik}=\left|\left\{j\mid\left(i,j\right)\in \mathcal{M}_{k}\right\}\right|$. Similarly, since $\psi\left(\varepsilon\right)=Y$ and $\left\{\widehat{\varepsilon}^{(1)},\ldots,\widehat{\varepsilon}^{(r)}\right\}$ is a basis of $\ker\left( A_{\mathcal{M}_{\mathcal{C}}(X)}\right)$. We conclude that $Y\in\textup{span}\ \mathcal{Y}_r$, by the construction of $\psi$ and the basis $\left\{\varepsilon^{(1)},\ldots,\varepsilon^{(d)}\right\}$ of $\mathbb{R}^d$.
\end{proof}

%-----------------------------------------------------------%

\begin{corollary}
Let $\mathcal{C}$ be a top-connected set of relations, $X$ a $\mathcal{C}$-pattern, and $\mathcal{M}_{\mathcal{C}}(X)$ its  associated tiling. Set $\lambda=(x_{n1},\ldots,x_{nn})$, and $\mu=(w_{1}(X),\ldots,w_{n}(X))$. Suppose that $\mathcal{M}_{\mathcal{C}}(X)$ has at least one $\Lambda_1$-free tile. Then
	\begin{itemize}
	\item[(i)] The dimension of the minimal face of $P^{+}_{\mathcal{C}}(\lambda)$ containing $X$ is equal to the number of $\Lambda_1$-free tiles in $\mathcal{M}_{\mathcal{C}}(X)$.
	\item[(ii)]  The dimension of the minimal face of $P^{+}_{\mathcal{C}}(\lambda,\mu)$  containing $X$ is equal to the dimension of the kernel of $A_{\mathcal{M}_{\mathcal{C}}(X)}$.
\end{itemize}

\end{corollary}
\begin{proof}
Follows directly from Theorem \ref{thm:dimI} and Lemma \ref{positive}.
\end{proof}

%-----------------------------------------------------------%

\begin{remark}
In general it is not true that the dimension of the minimal face of $P^{+}_{\mathcal{C}}$ containing $X$ is equal to the number of tiles in $\mathcal{M}_{\mathcal{C}}(X)$. In fact, under the conditions of Remark \ref{contraexemplo PC}, $X$ is a vertex for $P^{+}_{\mathcal{C}}$, but $\mathcal{M}_{\mathcal{C}}(X)$ has one tile.
\end{remark}

%===========================================================%
					%ACKNOWLEDGEMENTS%
%===========================================================%

\medskip

\textbf{Acknowledgements}. 
G.B. thanks the Institute of Mathematics and Statistics of University of S\~ao Paulo (IME -- USP), where part of this work has been done. G.B. was partially supported by the Coordena\c{c}\~{a}o de Aperfei\c coamento de Pessoal de N\'ivel Superior -- Brasil (CAPES) --  Finance Code 001. L.E.R. was supported by FAPESP grant 2018/17955-7.

%===========================================================%
						%REFERENCES%
%===========================================================%

\end{document}